\documentclass[10pt,letterpaper,dvips]{article}

 \usepackage{amssymb,latexsym,amsmath,amsthm}

\usepackage{fullpage}

\newtheorem*{thm*}{Theorem}

\newtheorem{thm}{Theorem}

\newtheorem{exam}{Example}

\newtheorem{lemma}{Lemma}

\newtheorem{remark}{Remark}

\newtheorem{cor}{Corollary}

\newtheorem{prop}{Proposition}

\begin{document}

\title{Some elliptic problems involving the gradient  on general bounded and exterior domains}




\author{A. Aghajani\thanks{School of Mathematics, Iran University of Science and Technology, Narmak, Tehran, Iran. Email: aghajani@iust.ac.ir.} \and C. Cowan\thanks{Department of Mathematics, University of Manitoba, Winnipeg, Manitoba, Canada R3T 2N2. Email: craig.cowan@umanitoba.ca. Research supported in part by NSERC.} 
}



\maketitle


\def\d{ \partial_{x_j} }
\def\Na{{\mathbb{N}}}

\def\Z{{\mathbb{Z}}}

\def\IR{{\mathbb{R}}}

\newcommand{\E}[0]{ \varepsilon}

\newcommand{\la}[0]{ \lambda}

\newcommand{\pa}{\partial}

\newcommand{\s}[0]{ \mathcal{S}}

\newcommand{\AO}[1]{\| #1 \| }

\newcommand{\BO}[2]{ \left( #1 , #2 \right) }

\newcommand{\CO}[2]{ \left\langle #1 , #2 \right\rangle}

\newcommand{\R}[0]{ \IR\cup \{\infty \} }

\newcommand{\co}[1]{ #1^{\prime}}

\newcommand{\p}[0]{ p^{\prime}}

\newcommand{\m}[1]{   \mathcal{ #1 }}

\newcommand{ \W}[0]{ \mathcal{W}}

\newcommand{ \A}[1]{ \left\| #1 \right\|_H }

\newcommand{\B}[2]{ \left( #1 , #2 \right)_H }

\newcommand{\C}[2]{ \left\langle #1 , #2 \right\rangle_{  H^* , H } }

 \newcommand{\HON}[1]{ \| #1 \|_{ H^1} }

\newcommand{ \Om }{ \Omega}

\newcommand{ \pOm}{\partial \Omega}

\newcommand{\D}{  C_c^{0,1}(\IR^N)   }

\newcommand{\DP}{ \mathcal{D}^{\prime} \left( \Omega \right)  }

\newcommand{\DPP}[2]{   \left\langle #1 , #2 \right\rangle_{  \mathcal{D}^{\prime}, \mathcal{D} }}

\newcommand{\PHH}[2]{    \left\langle #1 , #2 \right\rangle_{    \left(H^1 \right)^*  ,  H^1   }    }

\newcommand{\PHO}[2]{  \left\langle #1 , #2 \right\rangle_{  H^{-1}  , H_0^1  }}

 \newcommand{\HO}{ H^1 \left( \Omega \right)}

\newcommand{\HOO}{ H_0^1 \left( \Omega \right) }

\newcommand{\CC}{C_c^\infty\left(\Omega \right) }

\newcommand{\N}[1]{ \left\| #1\right\|_{ H_0^1  }  }

\newcommand{\IN}[2]{ \left(#1,#2\right)_{  H_0^1} }

\newcommand{\INI}[2]{ \left( #1 ,#2 \right)_ { H^1}}

\newcommand{\HH}{   H^1 \left( \Omega \right)^* }

\newcommand{\HL}{ H^{-1} \left( \Omega \right) }

\newcommand{\HS}[1]{ \| #1 \|_{H^*}}

\newcommand{\HSI}[2]{ \left( #1 , #2 \right)_{ H^*}}

\newcommand{\WO}{ W_0^{1,p}}
\newcommand{\w}[1]{ \| #1 \|_{W_0^{1,p}}}

\newcommand{\ww}{(W_0^{1,p})^*}

\newcommand{\Ov}{ \overline{\Omega}}
\newcommand{\underu}{\underline{u}}
\newcommand{\underv}{\underline{v}}
\newcommand{\overu}{\overline{u}}


\begin{abstract}  In this article we consider the existence of positive singular solutions on bounded domains and also classical solutions on exterior domains.   First we consider positive singular solutions of the following problems:

\begin{equation}  \label{eq_abst_1}-\Delta u    = (1+g(x)) | \nabla u|^p  \qquad \mbox{ in }  B_1, \qquad  u = 0 \mbox{ on } \;\;  \partial B_1, \qquad \mbox{ and}
\end{equation} 

\begin{equation} \label{eq_abst_2}
-\Delta u    =  | \nabla u|^p  \qquad \mbox{ in }  \Omega, \qquad  u = 0 \mbox{ on } \;\;  \pOm.
\end{equation}

In the first problem $B_1$ is the unit ball in $ \IR^N$ and in the second $\Omega$ is a bounded smooth domain in $ \IR^N$.  In both cases we assume $ N \ge 3$, $ \frac{N}{N-1}<p<2$ and in the first problem we assume $ g \ge 0$ is a H\"older continuous function with $g(0)=0$.     We obtain positive singular solutions in both cases. \\
We also consider (\ref{eq_abst_2}) in the case of $\Omega$ an exterior domain $ \IR^N$ where $N \ge 3$ and $ p >\frac{N}{N-1}$.  We prove the existence of a bounded positive classical solution of (\ref{eq_abst_2}) with the additional property that $ \nabla u(x) \cdot x>0$ for large $|x|$. 
\end{abstract}


\section{Introduction}

In this work we are interested in obtaining positive singular solutions of  
\begin{eqnarray} \label{eq_intro}
 \left\{ \begin{array}{lcl}
\hfill   -\Delta u    &=& (1+g(x))| \nabla u|^p  \qquad \mbox{ in }  B_1 \backslash \{0\},  \\
\hfill u &=& 0 \hfill \mbox{ on }  \partial B_1,
\end{array}\right.
  \end{eqnarray}
where $ p>1$ and  $B_1$ is the unit ball centered at the origin in $ \IR^N$.   Here $ g \ge 0$ is a H\"older continuous function with $ g(0)=0$.     We also consider the existence of positive singular solutions of 
\begin{eqnarray} \label{eq_intro_2}
 \left\{ \begin{array}{lcl}
\hfill   -\Delta u    &=& | \nabla u|^p  \qquad \mbox{ in }  \Omega,  \\
\hfill u &=& 0 \hfill \mbox{ on }  \pOm,
\end{array}\right.
  \end{eqnarray}
where $\Omega$ a bounded smooth domain in $\IR^N$.  Suppose $u$ is a classical solution of (\ref{eq_intro_2}),  then we can rewrite the equation as $ -\Delta u -b(x) \cdot \nabla u=0$ in $\Omega$ with $u=0$ on $ \pOm$ where $ b(x):= | \nabla u|^{p-2} \nabla u \in L^\infty$ and then apply the maximum principle to see $u=0$.   So the only hope of finding a nonzero solution of either problem is to find a singular solution.  
We also consider (\ref{eq_intro_2}) in the case of exterior domains.

  We now state our main theorems. 
   \begin{thm} \label{main_non_lin} (Bounded domain problems) \begin{enumerate} 
   \item Suppose $N \ge 3$, and $ \frac{N}{N-1}<p<2$ and $ g \ge 0$ is a H\"older continuous function with $g(0)=0$.  Then there exists an infinite number of positive singular solutions $u^t$ (indexed by $t$ for large $t$) of (\ref{eq_intro}) which blows up at the origin.  Moreover $u^t \rightarrow 0$ uniformly away from the origin. 
   \item   Let $ x_0 \in \Omega$ where $\Omega$ is a bounded domain with smooth boundary in $\IR^N$. Suppose  $p$ and $N$ satisfy the same restrictions as part 1 of the theorem.    Then there exists an infinite number of positive singular solution $u^t$ (indexed by $t$ for large $t$) of (\ref{eq_intro_2}) which blows up at $ x_0$ and is a classical solution away from $x_0$.   Moreover $u^t \rightarrow 0$ uniformly away from $x_0$. 
   
   \end{enumerate} 
    \end{thm}
    
    \begin{thm} \label{exterior_thm}  (Exterior domain problem) Suppose $N \ge 3$, $ \Omega$ is an exterior domain in $ \IR^N$ with smooth boundary and $ p> \frac{N}{N-1}$.   Then there is an infinite number of positive classical solutions of (\ref{eq_intro_2}) (say $u^t$ for large $t$) which satisfy 
$ \nabla u^t(x) \cdot x >0$.   In fact for large $t$ we have 
\[ \lim_{|x| \rightarrow \infty} \left( |x|^{N-2} ( x \cdot \nabla u^t(x)) - \frac{1}{t^\frac{1}{p-1}} \right)=0.\] 
    
    \end{thm}

We begin by looking  at a family of  explicit positive radial solutions on the unit ball centred at the origin which is taken from \cite{ACL2018}. 
\begin{exam} \label{example}  (\cite{ACL2018}) Let $B_1$ denote the unit ball centered at the origin in $ \IR^N$ for $N \geq 3$. Then for $ \frac{N}{N-1}<p<2$ we  define $\alpha:= (p-1)(N-1)$, $\beta:= \frac{p-1}{\alpha-1}, \sigma:=\frac{2-p}{p-1} $ and note $ \alpha >1$. Then 
\[ u_t(r) := \int_r^1 \frac{dy}{(\beta y+t y^\alpha  )^{1/(p-1)}}, \qquad t> -\beta,\] 
  is a positive singular solution of (\ref{eq_intro}) in the case of $g=0$.  
  \end{exam} 
  
  \begin{remark} \begin{enumerate}
      \item \textbf{The parameters.} 
  For the remaining sections of this work that deal with results on bounded domains we impose the parameter values from Example \ref{example}.   This includes all of the material in the Introduction also.  
  
  \item \textbf{The exterior problem.}  In Section \ref{exterior}, where we deal with exterior domains, some of the parameters will differ.  The crucial difference there will be that value of $\sigma$.   We will indicate the new values of the parameters in Section \ref{exterior}.   For an explicit solution of the exterior problem on $B_1^c$ see Example \ref{exam_ext}.
      
  \end{enumerate}
  \end{remark}

  \begin{remark} \begin{enumerate} \item  In a previous work (see \cite{ACL2018}) we linearized around $ u_t$ with $ t=0$ (whose linearized operator is given by $L_0$) to obtain solutions of perturbations of (\ref{eq_intro_2}) in the case of $\Omega = B_1$.   This allowed us to obtain singular solutions for (\ref{eq_intro_2}) for domains which are small perturbations of the unit ball.  It would also allow us to obtain solutions of (\ref{eq_intro}) in the case of $g$ satisfying a smallness condition.   This was also done for systems and a $p$-Laplace version, see  \cite{Cowanr, Cowanr1}.   The main new ingredient in the current work is to  linearize around the solution $u_t$ on the unit ball for $ t$ large.     This solution is no longer scale invariant and it is exactly this that allows us to remove any smallness condition on $g$ and in the case of general domains we don't need to consider perturbations of the ball.    See \cite{MP} Remark 3 for a similar statement.   
  \item Example \ref{example} is only one range of $p$ taken from an example in \cite{ACL2018}.  Many of the results here on bounded domains can be extended to the other ranges of $p$. 
  
  \end{enumerate}

\end{remark}

\subsection{Background}  A well studied problem is the existence versus non-existence of positive solutions of the Lane-Emden equation given by
\begin{eqnarray} \label{lane}
 \left\{ \begin{array}{lcl}
\hfill   -\Delta u    &=& u^p  \qquad \mbox{ in }  \Omega,  \\
\hfill u &=& 0 \hfill \mbox{ on }  \pOm,
\end{array}\right.
  \end{eqnarray} where $1<p$ and $ \Omega$ is a bounded domain in $ \IR^N$ (where $N \ge 3$) with smooth boundary.   In the subcritical case $ 1<p<\frac{N+2}{N-2}$  the problem is very well understood and $H_0^1(\Omega)$ solutions are classical solutions; see \cite{gidas}.  In the case of $ p \ge \frac{N+2}{N-2}$  there are no classical positive solutions in the case of the domain being star-shaped; see \cite{POHO}.  In the case of non star-shaped domains much less is known; see for instance \cite{Coron, M_1, M_2, M_3, Passaseo}.   In the case of $ 1<p<\frac{N}{N-2}$ ultra weak solutions (non $H_0^1$ solutions) can be shown to be classical solutions.  For $ \frac{N}{N-2}<p< \frac{N+2}{N-2}$  one cannot use elliptic regularity to show ultra weak solutions are classical.  In particular in \cite{MP} for a general bounded domain in $ \IR^N$  they construct singular ultra weak solutions with a prescribed singular set, see the book  \cite{Pacard} for more details on this. 


We now return to (\ref{eq_intro}).  The first point is that it is a non variational equation and hence there are various standard tools which are not available anymore.  The case $0<p<1$ has been studied in \cite{ACL}.       Some relevant monographs for this work include \cite{GT,GhRa2,STRUWE}.
Many people have studied boundary blow up  versions of (\ref{eq_intro}) where one removes the minus sign in front of the Laplacian;  see for instance  \cite{LaLi,ZZha}.
 See \cite{ACL, ABTP,ACTMOP,ACJT,ACJT2,BBM,BHV,BHV2, BHV3,ChC,FeMu,FePoRa,GPS,GPS2,GMP,GrTr,PS,Lions,MN,Ng,NV} for more results on equations similar to (\ref{eq_intro}). In particular, the interested reader is referred to \cite{Ng} for recent developments and a bibliography of significant earlier work, where the author studies isolated
singularities at $0$ of nonnegative solutions of the more general quasilinear equation
\[\Delta u=|x|^\alpha u^p+|x|^\beta |\nabla u|^q~~in~~\Omega\setminus\{0\},\]
where $\Omega\subset R^N$ ($N>2$) is a $C^2$ bounded domain containing the origin $0$, $\alpha>-2$, $\beta>-1$ and $p,q>1$, and provides a full classification of positive solutions  vanishing on $\partial \Omega$ and the removability of isolated singularities. \\

Before outlining our approach we mention that our work is heavily inspired by the works
\cite{pert_ori,MP,Pacard,davila,davila_fast,pert_wei, small_hole}.  Many of these works consider variations of  $-\Delta u = u^p$ on the full space or an exterior domain.  Their approach is to find an approximate solution and then to linearize around the approximate solution to find a true solution.   This generally involves a very detailed linear analysis of the linearized operator associated with approximate solution and then one applies a fixed point argument to find a true solution.

\subsection{Outline of approach.}

To give a brief outline of our approach we consider (\ref{eq_intro}), which is the cleanest case to consider since there are no cut-off functions needed.    We look for solutions of the form $u(x)= u_t(x) + \phi(x)$ (where $ \phi$ is the unknown and where we will end up taking  $ t$ large).  For $u$ to satisfy (\ref{eq_intro}) it is sufficient that $\phi$ satisfy
\begin{eqnarray} \label{eq_phi_intro}
 \left\{ \begin{array}{lcl}
\hfill  L_t(\phi)   &=& g(x)| \nabla u_t + \nabla \phi|^p + \left\{ | \nabla u_t + \nabla \phi|^p - | \nabla u_t|^p - p | \nabla u_t|^{p-2}  \nabla u_t \cdot \nabla \phi \right\}\qquad \mbox{ in }  B_1 \backslash \{0\},  \\
\hfill \phi &=& 0 \hfill \mbox{ on }  \partial B_1,
\end{array}\right.
  \end{eqnarray}  where 
  \[ L_t(\phi):=-\Delta \phi - p| \nabla u_t|^{p-2} \nabla u_t \cdot \nabla \phi,\] which is just the linearized operator associated with the solution $u_t$ of the unperturbed equation.  A computation shows that we have the explicit formula 
  \[ L_t(\phi)(x)=-\Delta \phi(x) + \frac{ p x \cdot \nabla \phi(x)}{ \beta |x|^2+ t |x|^{\alpha+1}}.\]    We now define the norms we will use for (for the problem on $B_1$); 
  \[ \|f \|_Y:=\sup_{B_1} |x|^{\sigma+2} |f(x)|, \quad \| \phi \|_X:= \sup_{B_1} \left\{ |x|^\sigma | \phi(x)| + |x|^{\sigma+1} | \nabla \phi(x)| \right\},\]  and we denote $Y,X$ as the appropriate spaces; for the space $X$ we impose the boundary condition $ \phi=0$ on $ \partial B_1$.       To obtain a solution $\phi$ of (\ref{eq_phi_intro}) we will find a fixed point of the following mapping: $ T_t(\phi)=\psi$ where 
\begin{eqnarray} \label{eq_phi_intro_mapp}
 \left\{ \begin{array}{lcl}
\hfill  L_t(\psi)   &=& g(x)| \nabla u_t + \nabla \phi|^p   + \left\{ | \nabla u_t + \nabla \phi|^p - | \nabla u_t|^p - p | \nabla u_t|^{p-2}  \nabla u_t \cdot \nabla \phi \right\}\qquad \mbox{ in }  B_1 \backslash \{0\},  \\
\hfill \psi &=& 0 \hfill \mbox{ on }  \partial B_1.
\end{array}\right.
  \end{eqnarray}  In the end we will show that $T_t$ is a contraction on $B_R$ (the closed ball of radius $R$ centred at the origin in $X$) and hence we can apply Banach's fixed point theorem.    This will give the existence of $ \phi$ and then we will argue that $u(x)= u_t(x) + \phi(x)$ is positive in $B_1$.     A crucial point is that $u_t$ converges to zero outside of the origin and hence we will be able to view the term $g(x)| \nabla u_t + \nabla \phi|^p$ as small since $g(0)=0$;  which allows us not to impose any smallness assumption on $g$.    
  
 \subsubsection{Outline of article.} 
  
  The approach outlined above makes up Section \ref{linear_ball_sect}, which contains the linear theory, and Section \ref{section_fixed_ball}, which contains the fixed point argument.   
  
  In Section \ref{gen_dom_sec}  we consider (\ref{eq_intro_2}) on bounded domains.    The needed linear theory here will come from the linear theory on $B_1$ coupled with a gluing argument.  Section \ref{gen_dom_sec} also contains the needed fixed point argument, which is more involved then it was for (\ref{eq_intro}). 
  
 In Section \ref{exterior} we examine (\ref{eq_intro_2}) in the case of exterior domains.   Here the needed linear theory can come via perturbing the Laplacian on a general exterior domain.  The theory here involves a different choice of weight $ \sigma$ then on the bounded domain case.     The fixed point argument here follows essentially the fixed point arguments used in Section \ref{gen_dom_sec}.

  \section{The linear operator $L_t$ on $B_1$}  \label{linear_ball_sect}
  In this section we examine the linear operator $L_t$ on $B_1$ and we now state our main result regarding this.
  \begin{prop}  \label{main_linear} There is some $C>0$ and $t_0$ (large) such that for all  $f \in Y$ there is some $ \phi \in X$ such that 
  \begin{eqnarray} \label{lin_f}
 \left\{ \begin{array}{lcl}
\hfill   L_t(\phi)  &=& f \qquad \mbox{ in }  B_1 \backslash \{0\},  \\
\hfill  \phi &=& 0 \hfill \mbox{ on }  \partial B_1.
\end{array}\right.
  \end{eqnarray}  Moreover one has the estimate $ \| \phi \|_X \le C \|f\|_Y$. 
   \end{prop}

  One should note that, at least formally,  that $ \partial_t u_{t}(r)|_{t=1}$ is in the kernel of $L_t$ on $B_1$. In fact this is the case and  if we set 
  \[ \psi_t(r):= -\partial_t u_{t}(r) = \frac{1}{p-1} \int_r^1 \frac{y^\alpha}{\left( \beta y + t y^\alpha \right)^\frac{p}{p-1}} dy,\] 
   then $ \psi_t \in X$ and satisfies $L_t(\psi_t)=0$ in $B_1 \backslash\{0\}$ with $ \psi_1=0$ on $ \partial B_1$.     \\ 
   
   \noindent \emph{Spherical harmonics.}  Consider the eigenpairs $ (\psi_k,\lambda_k)$ of the Laplace-Beltrami operator $\Delta_\theta=\Delta_{S^{N-1}}$ on $S^{N-1}$ which satisfy 
   \begin{equation*} 
   -\Delta_\theta \psi_k(\theta) = \lambda_k \psi_k(\theta), \quad \mbox{  in } \; \theta \in S^{N-1},
   \end{equation*} which we normalize $ \| \psi_k \|_{L^2(S^{N-1})}=1$.    Note that $ \psi_0=1$, $ \lambda_0=0$ (multiplicity $1$);  $ \lambda_1 = N-1$ (multiplicity $N$);  $ \lambda_2 =2N$.  \\
   
   Given a $ f \in Y$,  $ \phi \in X$ we write 
   \[ f(x) = \sum_{k=0}^\infty b_k(r) \psi_k(\theta), \qquad \phi(x)=\sum_{k=0}^\infty a_k(r) \psi_k(\theta),\]  and note $ a_k(1)=0$ after considering the boundary condition of $\phi$.    A computation shows that  $ \phi$ satisfies (\ref{lin_f}) provided $a_k$ satisfies 
   \begin{equation} \label{non-homo-ode}
   -a_k''(r)-\frac{(N-1) a_k'(r)}{r} + \frac{\lambda_k a_k(r)}{r^2} + \frac{ p a_k'(r)}{\beta r + t r^\alpha} = b_k(r), \quad \mbox{ for }  0<r<1,
   \end{equation} with $ a_k(1)=0$.    Since we already developed a theory for  the linear operator $L_0$ in \cite{ACL2018} we prefer to utilize some continuation arguments to obtain results for $L_t$.    This will work sufficiently well except one needs to be a bit careful since we recall that $\psi_t$ is in the kernel of $L_t$.     Noting that $ \psi_t$ is is radial one sees this solves the homogenous version of (\ref{non-homo-ode}) when $k=0$.  For the $k=0$ mode we will need to solve (\ref{non-homo-ode}) directly, see Lemma \ref{k=0 mode}. We now define some spaces to remove this problematic $k=0$ mode.  
Define the closed subspaces of  $X$ and $Y$ by   
   \[ X_1:= \left\{ \phi \in X:   \phi(x)=\sum_{k=1}^\infty a_k(r) \psi_k(\theta) \right\}, \qquad Y_1:= \left\{ f \in Y:  f(x) = \sum_{k=1}^\infty b_k(r) \psi_k(\theta) \right\},\]  note the sums start at $k=1$ and not $k=0$.  
   We begin by stating a few results from \cite{ACL2018}.    In what follows we will be working on $B_R$ or $ \IR^N$ and the spaces $X$ and $X_1$ are obvious extensions of the definitions to these more general settings.

   \begin{lemma} (\cite{ACL2018}).   
   \begin{enumerate}  \item  Let $ 0<R \le \infty$ and  suppose $ \phi \in X_1$ is such that $L_0(\phi)=0$ in $B_R \backslash \{0\}$ with $ \phi=0$ on  $ \partial B_R$ in the case of $R$ finite.      Then $ \phi=0$.  
   \item Proposition \ref{main_linear}  holds if one replaces $L_t$ with $L_0$. 
   \end{enumerate} 
   \end{lemma}

   \begin{proof} For the convenience of the reader we prove part 1. 
    We write $ \phi(x)=\sum_{k=1}^\infty a_k(r) \psi_k(\theta)$ and so $a_k$ satisfies 
  \[ a_k''(r) + \frac{(N-1) a_k'(r)}{r} - \frac{pa_k'(r)}{\beta r} - \frac{\lambda_k a_k(r)}{r^2}=0, \quad \mbox{ for } \;  0<r<R,\] with $ a_k(R)=0$ in the case of $R<\infty$.    Also we have $ \sup_{0<r<R} \left\{ r^\sigma |a_k(r)| r^\sigma + r^{\sigma+1} | a'_k(r)| \right\} <\infty$. Note this ode is of Euler form and hence its solutions  are $a_k(r) = C_k r^{\gamma_k^+} + D_k r^{\gamma_k^-}$ for some $C_k,D_k \in \IR$ where $ \gamma_k^\pm$ are the roots of 
  \[ \gamma^2 +(N-2 - \frac{p}{\beta}) \gamma - \lambda_k=0,\]  which are given by 
  \[ \gamma_k^{\pm}= \frac{-(N-2- \frac{p}{\beta})}{2} \pm  \frac{ \sqrt{ (N-2 - \frac{p}{\beta})^2 + 4 \lambda_k}  }{2}.\] 
  
  A computation shows that $ \gamma_1^-+\sigma=-1$ and so we have $\gamma_k^- +\sigma \le -1$ for $ k \ge 1$.      We first consider the case where $0<R<\infty$.    To satisfy $a_k(R)=0$ we see there is some $ \alpha_k(R)  \neq 0$ and $ \tilde{C}_k \in \IR$ such that $ a_k(r)= \tilde{C}_k \left( \alpha_k(R) r^{\gamma_k^+} + r^{\gamma_k^-} \right)$.  Now since $k \ge 1$ we have $ \gamma_k^-+\sigma \le -1 <0$ and so to have $a_k$ in the appropriate space we must have $\tilde{C}_k=0$.     Now we consider the case of $R=\infty$. In this case we have $a_k(r) = C_k r^{\gamma_k^+} + D_k r^{\gamma_k^-}$ and provided $ \gamma_k^+ +\sigma \neq 0$ and $ \gamma_k^-+\sigma \neq 0$ we can send $ r \rightarrow 0,\infty$ to see we must have $C_k=D_k=0$ for $a_k$ to be in the required space.   So to complete the proof we only need to verify $ \gamma_k^+ +\sigma \neq 0$. A computation shows that 
  \[ \sigma+\gamma_1^+= \frac{(N-1)p^2 +p(-2N+1)+N+1}{p-1}>0,\] and the desired result follows by monotonicity in $k$.
  \end{proof}

  

 \begin{lemma} \label{kernel_t} (Kernel of $L_t$ in $X_1$)  Let $0<R \le \infty$, $ t \in (0,\infty]$ and $ \phi \in X_1$  with $L_t(\phi)=0$ in $B_R \backslash \{0\}$ with $ \phi=0$ on $ \partial B_R$ in the case of $R$ finite.  Then $ \phi=0$. 
 \end{lemma} 
  
   \begin{proof}   Suppose $ R,t,\phi$ as in the hypthosis.  Further we suppose $ 0<t<\infty$  since  $L_\infty=-\Delta$, and this result is well known for the Laplacian.     We write $ \phi(x)=\sum_{k=1}^\infty a_k(r) \psi_k(\theta)$; note there is no $k=0$ mode since $ \phi \in X_1$.    Then $ a_k$ satisfies 
  \begin{equation} \label{ode_kernel_t}
  -\Delta a_k(r) + \frac{ \lambda_k}{r^2} a_k(r) + \frac{p a_k'(r)}{\beta r + t r^\alpha}=0, \quad 0<r<R,
  \end{equation}and in the case of $ R<\infty$ we have $a_k(R)=0$.   Moreover there is some $C_k>0$ such that \[ \sup_{0<r<R} \left\{ r^\sigma |a_k(r)|+ r^{\sigma+1} |a_k'(r)| \right\} \le C_k.\]  
  Fix $ k \ge 1$ and we set $ w(\tau):=r^\sigma a_k(r)$ where $ \tau=\ln(r)$.  Then a computation shows that $ w=w(\tau)$ satisfies 
  \begin{equation} \label{wave_k}
  0=w_{\tau \tau} + g(\tau) w_\tau + C_k(\tau) w,  \qquad \tau \in (-\infty, \ln(R)), 
  \end{equation}
  where 
  \[ g(\tau) = N-2-2\sigma - \frac{p}{\beta+ t e^{(\alpha-1) \tau}} \] 
   \[ C_k(\tau):= - \lambda_k +  \frac{p \sigma}{ \beta+ t e^{(\alpha-1) \tau}} - \sigma(N-2-\sigma).\]

  We now claim that one has the improved decay estimate; $ r^\sigma |a_k(r)|  \rightarrow 0$ as $ r \rightarrow 0$ and in the case of $ R=\infty$ that we have $ r^\sigma |a_k(r)| \rightarrow 0$ as $ r \rightarrow \infty$.      For the moment we assume we have the claim.   Then note this gives that $ w \rightarrow 0$ as $ \tau \rightarrow -\infty$ and in the case of $ R=\infty$ we have the same result when $ \tau \rightarrow \infty$.  
  
    By multiplying by $ -1$, if needed, we can assume that if $ w \neq 0$ (and since $ w(-\infty)=w( \ln(R))=0$ we can suppose there is some $ \tau_0 \in (-\infty, \ln(R))$ such that $ w(\tau_0)=\max w >0$.   Then we have $ w_{\tau \tau}(\tau_0) \le 0$ and $ w_\tau(\tau_0)=0$ and hence from the equation we get $ C_k(\tau_0) w(\tau_0) = - w_{\tau \tau}(\tau_0) \ge 0$.  From this we see that we must have $C_k(\tau_0) \ge 0$.  Using the monotonicity of $C_k$ in $ \tau$ and $k$ we see that for all $ \tau \in (-\infty, \ln(R))$ we have \[ C_k(\tau) \le C_k(-\infty) \le C_1(-\infty) = -(N-1)+ \frac{p \sigma}{\beta} - \sigma (N-2-\sigma)\] and this quantity can be seen to be negative after considering the restrictions on $p$.    Hence we must have $ w=0$ and hence $a_k=0$ for all $ k \ge 1$. 
  We now prove the the claimed decay estimates.  Fix $k \ge 1$ and set $ a(r)=a_k(r)$ so we have 
\[ -\Delta a(r) + \frac{\lambda_k a(r)}{r^2}+ \frac{p a'(r)}{\beta r + t r^\alpha} = 0,  \qquad \mbox{ in } \; 0<r<R,\] with $ a(R)=0$.   Suppose the claim is false.  Then there is some $r_m \rightarrow 0$ such that $ r_m^\sigma |a(r_m)| \ge \E_0>0$.  Define the rescaled functions $ a^m(r):= r_m^\sigma a(r_m r)$ and note $ |a_m(1)| \ge \E_0$ and $ r^\sigma |a^m(r)| \le C$. A computation shows that 
\[ -\Delta a^m(r) + \frac{\lambda_k a^m(r)}{r^2} + \frac{ (a^m)'(r)}{ \beta r+ t r_m^{\alpha-1} r^\alpha} =0, \quad \mbox{ in } \; 0<r<\frac{R}{r_m}.\]  Passing to the limit we can find some $a^\infty \neq 0$ with $ r^{\sigma} |a^\infty(r)| + r^{\sigma+1}| (a^\infty)'(r)| \le C$ which satisfies $L_0(a^\infty)=0$ in $0<r<\infty$, but this contradicts our earlier theory on $L_0$.   In the case of $ R=\infty$ the proof is similar,  but the limiting equation is $ L_\infty(a^\infty)=0$ in $ 0<r<\infty$. 
\end{proof}

  \begin{prop} \label{higher_modes_linear}(Linear theory for $L_t$ on $X_1$)  There is some $C>0$ and $ t_0$ (large) such that  for all $ t \ge t_0$ and $ f \in Y_1$ there is some $ \phi \in X_1$ which satisfies (\ref{lin_f}).  Moreover one has the estimate $\| \phi \|_X \le C \|f\|_Y$. 
  \end{prop} 
  
  \begin{proof}  Since we already have a well developed theory regarding $L_0$ we will use a continuation argument to connect this to $L_t$.  For the continuation argument we need to define a new norm, 
  \[ \| \phi \|_{\widehat{X}}:=\sup_{B_1} \left\{ |x|^\sigma | \phi(x)| + |x|^{\sigma+1} | \nabla \phi(x)| + |x|^{\sigma+2} | \Delta \phi(x)| \right\},\] and we define the spaces $\widehat{X}$ accordingly and we set $ \widehat{X}_1$ to be the functions in $\widehat{X}$ with no $k=0$ mode.   We begin by showing that for each $ 0<t<\infty$  that we have the desired mapping properties;  but possibly the constant $C$ depends on $t$.  Later we show we can take $C$ independent of $t$ for large $t$;  really this result holds for all $ t \ge 0$  but we will only need it independent of $t$ for large $t$.      So fix $ 0<\gamma<\infty$ and consider  $ (t,\phi) \mapsto L_t(\phi)$ is continuous from $ [0,\gamma] \times \widehat{X}_1 $ to $Y_1$.  Additionally from our previous work \cite{ACL2018} we know that $L_0: \widehat{X}_1 \rightarrow Y_1$ is an isomorphism.    To prove $L_\gamma$ has the desired mapping properties is it sufficient to obtain bounds on $L_t$ for $ 0 \le t \le \gamma$.   So we suppose there is $ 0 \le t_m \le \gamma$ and $f_m \in Y_1,  \phi_m \in \widehat{X}_1$  such that $L_{t_m}(\phi_m)=f_m$ in $B_1 \backslash \{0\}$ with $ \phi_m=0$ on $ \partial B_1$  and $ \|f_m \|_Y \rightarrow 0$, $ \| \phi_m \|_{\widehat{X}}=1$.    To get a contradiction we will show that $ \| \phi_m\|_{\widehat{X}} \rightarrow 0$.    It will be sufficient to show that $ \sup_{B_1} | x|^{\sigma+1} | \nabla \phi_m(x)| \rightarrow 0$.  To see this note we can integrate the first order estimate to obtain the zero order estimate.    Also directly from the pde we get the second order estimate if we have the first order one.    So we suppose not; then there is some $ \E_0>0$ and $ x_m \in B_1 \backslash \{0\}$ such that $ |x_m|^{\sigma+1} | \nabla \phi_m(x_m)| \ge \E_0$.   Set $s_m:=|x_m|$ and now consider two cases (in that follows we are passing to subsequences if necessary):   (i) $s_m$ bounded away from zero,  \; (ii) $s_m \rightarrow 0$ (in both cases we assume $t_m \rightarrow t$). \\ 
  
  \noindent 
  \emph{Case (i).}   By elliptic theory $ \phi_m$ is bounded in $C^{1,\delta}_{loc}( \overline{B_1} \backslash \{0\})$ and converges in this space to some $ \phi$.   Since $s_m$ is bounded away from zero we see that $ \phi \neq 0$.   Additionally we have $L_t(\phi)=0$ in $B_1 \backslash \{0\}$ with $ \phi=0$ on $ \partial B_1$.   Also note $ \phi \in \widehat{X}_1$ and hence by our earlier kernel results we know $ \phi=0$, a contradiction. \\
  
  \noindent 
  \emph{Case (ii).}   Define $ \zeta_m(z):=s_m^\sigma \phi_m(s_m z)$ for $ |z|<\frac{1}{s_m}$ and note we have the bounds $ |z|^{\sigma} | \zeta_m(z)| + |z|^{\sigma+1} | \nabla \zeta_m(z)| \le 1$.   Define $z_m= s_m^{-1} x_m$  and note $ |z_m|=1$ and $ | \nabla \zeta_m(z_m)| \ge \E_0$.  A computation shows that 
  \begin{equation} \label{eq_1011}
  L_{t_m s_m^{\alpha-1}}(\zeta_m)= g_m(z):=s_m^{\sigma+2} f_m(s_m z) \mbox{ in } B_\frac{1}{s_m}, \qquad \zeta_m=0 \mbox{ on }  \partial B_\frac{1}{s_m}.
  \end{equation}   By elliptic estimates applied to an increasing sequence of annuli, and a suitable diagonal argument,  there is some $ \zeta $ such that $\zeta_m \rightarrow \zeta$ in $C^{1,\delta}_{loc}(\IR^N \backslash \{0\})$ and  note there is some $|z_0|=1$ (the limit of the $z_m$) such that $| \nabla \zeta(z_0)| \ge \E_0$ and hence $ \zeta \neq 0$.       But we also note that $L_0(\zeta)=0$ in $ \IR^N \backslash \{0\}$ and $\zeta$ satisfies the needed bounds to be able to apply our earlier Liouville results, hence $ \zeta=0$; which gives the needed contradiction.    \\
  
  So we have  shown that for each $ t \ge 0$ there is some $C_t$ such that we have the desired linear theory if we replace $C$ with $C_t$.   Now we show the $C_t$ can be taken independently of $t$.     Note that the above proof really shows the result could only fail in the case of $ t \rightarrow \infty$.       
  
    So we suppose the result is false;  so there is some $t_m \rightarrow \infty$,  $ f_m \in Y_1$, $ \phi_m \in X_1$  such that $L_{t_m}(\phi_m)=f_m $ in $B_1 \backslash \{0\}$ with $ \phi_m=0$ on $ \partial B_1$ with $ \|f_m \|_Y \rightarrow 0$ and $ \| \phi_m \|_X=1$.   As before there is some $x_m \in B_1 \backslash \{0\}$ such that $ |x_m|^{\sigma+1} | \phi_m(x_m)| \ge \E_0$.  Set $s_m:=|x_m|$ and we consider the cases: \\ (i) $s_m$ bounded away from zero, \; (ii) $s_m \rightarrow 0$. \\ 
    
    \noindent 
    \emph{Case (i).}  From the equation and compactness arguments we see we see there is some $\phi$ such that $ \phi_m \rightarrow \phi $ in $C^{1,\delta}( \overline{B_1} \backslash \{0\})$.  Since $ s_m$ bounded away from zero we see that $ \phi \neq 0$ and also note that $ \phi \in X_1$.     Additionally we can pass to the limit in the equation to see $L_\infty(\phi)=0$ in $B_1 \backslash \{0\}$ with $ \phi=0$ on $ \partial B_1$;  but this contradicts the earlier kernel results. \\
    
    \noindent 
    \emph{Case (ii).}  We now follow exactly the case (ii) from the finite $t$;  set $\zeta_m(z)= s_m^\sigma \phi_m(s_m z)$ and then note  $ |z|^{\sigma} | \zeta_m(z)| + |z|^{\sigma+1} | \nabla \zeta_m(z)| \le 1 $.   Define $z_m= s_m^{-1} x_m$  and note $ |z_m|=1$ and $ | \nabla \zeta_m(z_m)| \ge \E_0$.  As before $ \zeta_m$ satisfies (\ref{eq_1011}).   Again we use a compactness argument away from the origin and $ \infty$ to pass to the limit $ \zeta$  in $C^{1,\delta}_{loc}(\IR^N \backslash \{0\})$ and hence $ |\nabla \zeta(z_0)| \ge \E_0$ for some $|z_0|=1$ and   $ |z|^{\sigma} | \zeta(z)| + |z|^{\sigma+1} | \nabla \zeta(z)| \le 1 $ in $ \IR^N \backslash \{0\}$.  Moreover $ \zeta$ satisfies $L_\gamma(\zeta)=0$ in $ \IR^N \backslash \{0\}$ where $\gamma= \lim_{m} t_m s_m^{\alpha-1} \in [0,\infty]$.  In all cases we can apply our earlier kernel results to obtain a contradiction. 

 \end{proof}
 
 \begin{lemma} \label{k=0 mode}($k=0$ mode for $L_t$)

 \end{lemma}

 \begin{proof}  Consider (\ref{non-homo-ode}) in the case of $k=0$ and to indicate the dependence on $t$ we will write $a_t(r)$. Assume $ \sup_{0<r<1} |b(r)| r^{\sigma+2} \le 1$.     A computation shows an integrating factor associated with the ode is given by 
 \[ \mu_t(r)=r^{N-1} e^{\int_r^1 \frac{1}{\beta s + t s^\alpha} d s}=      r^{N-1 - \frac{p(\alpha-1)}{p-1}} \left(  \frac{\beta+t r^{\alpha-1}}{\beta+t} \right)^\frac{p}{p-1}.\]  
 We then obtain
 \[ \mu_t(r)a_t'(r) = a'_t(1)- \int_r^1 \mu_t(\tau) b(\tau) d\tau,~~~0<r\le1. \] 
We set
\[a'_t(1)=\int_{R_t}^1\mu_t(\tau) b(\tau) d\tau,\] 
where $R_t^{\alpha-1}t =1$. Then we get
\[a_t'(r)= \frac{1}{\mu_t(r)}\int_{R_t}^{r}\mu_t(\tau)b(\tau)d\tau,~~~0<r\le1. \] 
 and so we can write $a_t$ as 
 \[ a_t(r):= \int_r^1 \left( \frac{1}{\mu_t(s)} \int_{R_t}^s \mu_t(\tau) b(\tau) d \tau \right) ds,~~~0<r\le1. \]  and note $a_t(1)=0$.  The only thing left to check is that $a_t$ satisfies the desired bounds independent of $t$ for large $t$;  note this careful choice of $R_t$ is what gives the estimate.  Also note we only need to satisfy the first order estimate since we can integrate this to obtain the zero order estimate.   So writing out $a_t'(r)$ we see, using the equality $N-1 - \frac{p(\alpha-1)}{p-1}=\sigma-\alpha+2$, that 
 \[r^{\sigma+1}|a_t'(r)|\le r^{\alpha-1}\Big| \int_{R_t}^{r}\Big(\frac{\beta+t\tau^{\alpha-1}}{\beta+tr^{\alpha-1}}\Big)^{\frac{p}{p-1}}\frac{d\tau}{\tau^\alpha}\Big|\] and we now consider the two cases: (i) $ 0<r<R_t$, \;(ii) $ R_t <r<1$.  \\
 
 \noindent
 \emph{ Case (i).}  For $r<R_t$ we have
 
\begin{eqnarray*}
r^{\sigma+1}|a_t'(r)| & \le & r^{\alpha-1} \int_{r}^{R_t}\Big(\frac{\beta+t\tau^{\alpha-1}}{\beta+tr^{\alpha-1}}\Big)^{\frac{p}{p-1}}\frac{d\tau}{\tau^\alpha} \\
&\le &r^{\alpha-1}\Big(\frac{\beta+tR_t^{\alpha-1}}{\beta+tr^{\alpha-1}}\Big)^{\frac{p}{p-1}}\int_{r}^{R_t} \frac{d\tau}{\tau^\alpha} \\
&=& \Big(\frac{\beta+1}{\beta+tr^{\alpha-1}}\Big)^{\frac{p}{p-1}}\left( \frac{1-(\frac{R_t}{r})^{1-\alpha}}{\alpha-1}\right) \\
&\le &\Big( \frac{\beta+1}{\beta}\Big)^{\frac{p}{p-1}}\frac{1}{\alpha-1}.
\end{eqnarray*}
Thus we proved
\begin{equation}
r^{\sigma+1}|a_t'(r)|\le \Big( \frac{\beta+1}{\beta}\Big)^{\frac{p}{p-1}}\frac{1}{\alpha-1},~~for~~~0<r<R_t.
\end{equation}
 
 \noindent
 \emph{ Case (ii).}  For $r> R_t$ we write, using the inequality $(a+b)^q\le c_q (a^q+b^q)$ for $q>1$,
\begin{eqnarray*}
r^{\sigma+1} |a_t'(r)| &\le& \frac{r^{\alpha-1}}{(\beta+tr^{\alpha-1})^{\frac{p}{p-1}}} \int_{R_t}^r\frac{(\beta+t\tau^{\alpha-1})^{\frac{p}{p-1}}}{\tau^\alpha}d\tau \\
& \le & \frac{Cr^{\alpha-1}}{(\beta+tr^{\alpha-1})^{\frac{p}{p-1}}} \int_{R_t}^r \Big(\frac{1}{\tau^\alpha}+\frac{t^{\frac{p}{p-1}}}{\tau^{\alpha-\frac{p(\alpha-1)}{p-1}}}\Big) d\tau \\
&\le& \frac{C_1r^{\alpha-1}}{(\beta+tr^{\alpha-1})^{\frac{p}{p-1}}}
\Big(R_t^{1-\alpha}+t^{\frac{p}{p-1}}r^{\frac{\alpha-1}{p-1}}\Big),
\end{eqnarray*} 
where $C_1$ is a constant independent of $t$. Recall we have $tR_t^{\alpha-1}=1$, thus 
\[ r^{\sigma+1} |a_t'(r)|\le  \frac{C_1 tr^{\alpha-1}}{(\beta+tr^{\alpha-1})^{\frac{p}{p-1}}}+\frac{C_1 (tr^{\alpha-1})^{\frac{p}{p-1}}}{(\beta+tr^{\alpha-1})^{\frac{p}{p-1}}}  \le \frac{C_1}{(tr^{\alpha-1})^{\frac{1}{p-1}}}+C_1,\]
and since for $r\ge R_t$ we have $tr^{\alpha-1}\ge tR_t^{\alpha-1}=1$
we get
\begin{equation}
 r^{\sigma+1} |a_t'(r)|\le C_1+C_1=2C_1,~~\text{for}~~r\ge R_t.
\end{equation}
Combining case (i) and (ii) gives 
\[ r^{\sigma+1} |a_t'(r)|\le \max\Big\{\Big( \frac{\beta+1}{\beta}\Big)^{\frac{p}{p-1}}\frac{1}{\alpha-1}, 2C_1\Big\}.\]
\end{proof}

 
 \noindent 
 \textbf{Completion of the proof of Proposition \ref{main_linear}.}   Here we combine  Lemma \ref{k=0 mode} and Proposition \ref{higher_modes_linear} to complete the proof of Proposition \ref{main_linear}.   Let $ f \in Y$ and let $ \phi \in X$ satisfy (\ref{lin_f}) and we write $ f(x)=f_0(r)+ f_1(x)$, $ \phi(x)= \phi_0(r)+ \phi_1(x)$ where we have split off the $k=0$ mode and $ \phi_1 \in X_1, f_1 \in Y_1$.  Then we have 
 \begin{eqnarray*}
 \| \phi \|_X & \le & \| \phi_0 \|_X + \| \phi_1\|_X  \\ 
  & \le &  \| C \|f_0 \|_Y + C \|f_1 \|_Y 
  \end{eqnarray*} and hence if we can show there is some $D>0$ (independent of $f$) such that $ \| f_0 \|_Y + \|f_1\|_Y \le D \|f_0 + f_1 \|_Y$  then we would be done.   We suppose the result is false and hence for all $ m \ge 1$ there is some $f^m \in Y$ such that 
  \[1= \|f_0^m \|_Y + \|f_1^m\|_Y>m \|f^m\|_Y\] where we have also performed a normalization of $f^m$ and hence $ f^m \rightarrow 0$ in $Y$. Then note we have 
 $|S^{N-1}| f_0(r) = \int_{| \theta|=1} f^m(r \theta) d \theta$ and hence 
 \[ r^{\sigma+2} |f_0^m(r)| \le C \int_{|\theta|=1} r^{\sigma+2} |f^m(r \theta)| d \theta \le C_1 \|f^m \|_Y,\] and hence $ \|f_0^m\|_Y \rightarrow 0$.    Also note we have $ f_1^m(x)= f^m(x)- f_0^m(r)$ and hence $ \|f_1^m\|_Y \rightarrow 0$; a contradiction. 
 \hfill
 $\Box$

  \section{The fixed point argument for (\ref{eq_intro})} \label{section_fixed_ball}

\begin{lemma}  Suppose $ 1<p \le 2$.  Then there is some $ C=C_p$ such that  for all $x,y,z \in \IR^N$ one has
\begin{equation} \label{ineq_1}
0 \le |x+y|^p- |x|^p - p |x|^{p-2} x \cdot y \le C |y|^p,
\end{equation} 
\begin{equation} \label{ineq_2}
\Big| |x+y|^p-p |x|^{p-2} x \cdot y - |x+z|^p + p|x|^{p-2} x \cdot z \Big| \le C \left( |y|^{p-1} + |z|^{p-1} \right) |y-z|.
\end{equation} 

\begin{equation} \label{ineq_3}
\Big  |x+y|^p-|x+z|^p \Big| \le C \left( |y|^{p-1} + |z|^{p-1} + |x|^{p-1} \right) |y-z|.
\end{equation} 

\end{lemma}

We will need some asymptoptics of $u_t$.  So first note that 
\[ u_t'(r) =  \frac{-1}{ \left( \beta r+ t r^\alpha \right)^\frac{1}{p-1}}, \qquad \mbox{ and if we set } C_\beta:= \frac{1}{\beta^\frac{1}{p-1}},\] and hence 
\begin{equation} \label{asymp_1} 
| u_t'(r)| \le \min \left\{  \frac{C_\beta}{ r^\frac{1}{p-1}},    \frac{1}{ t^\frac{1}{p-1} r^{N-1}} \right\}, \qquad \mbox{ so } 
\end{equation} 
\begin{equation} \label{asymp_2} 
r^{\sigma+1}| u_t'(r)| \le \min \left\{  C_\beta,    \frac{1}{ t^\frac{1}{p-1} r^{N-2-\sigma} } \right\}.
\end{equation}   So we see for any $ t>0$ we have $ \lim_{r \rightarrow 0} r^{\sigma+1} u_t'(r)= -C_\beta$ and $u_t,u_t' $ converge uniformly to zero away from the origin.    In what follows $ B_R$ is the closed ball in $X$ centred at the origin with radius $R$.

\begin{lemma}(Into)  \label{into_lemma} 
\begin{enumerate} 
\item There is some $C>0$ such that for all  $0<R<1$, $0<\delta<1$, $ t>1$, $ \phi \in B_R \subset X$   one has 
\[ \| g| \nabla u_t + \nabla \phi|^p \|_Y \le C  \left( R^p + \sup_{|z|<\delta } |g(z)| + \frac{1}{t^\frac{p}{p-1} \delta^{(N-1)p-\sigma-2}} \right).\]  

\item There is some $C>0$ such that for all $ t>1$, $0<R<1$ and $ \phi \in B_R$ one has 
\[  \Big\|  | \nabla u_t+ \nabla \phi|^p - p | \nabla u_t|^{p-2} \nabla u_t\cdot \nabla \phi -| \nabla u_t|^p \Big\|_Y \le C R^p.\] 
\end{enumerate} 
\end{lemma}

\begin{proof}  \begin{enumerate} \item  Fix $ R,\delta,\phi$ as in the hypothesis and $ C$ will denote a changing constant that is independent of these parameters.   Set $ I_0:=|g(x)| |x|^{\sigma+2} | \nabla u_t + \nabla \phi|^p \le C |g(x)| |x|^{\sigma+2} \left\{ | \nabla u_t|^p + |\nabla \phi|^p \right\}$.  ALso note we have the estimates $ |x|^{\sigma+2} | \nabla \phi(x)|^p \le R^p$ and $r^{\sigma+2} | \nabla u_t(r)| \le C$.  The first step is to write $\sup_{B_1} |I_0|$ as a sup over $B_\delta$ and $ \delta <|x|<1$.  Doing this gives $\sup_{B_\delta} I_0 \le C \sup_{B_\delta}|g|$.   For the other portion we obtain 
\[ \sup_{\delta <|x|<1} I_0 \le C R^p + C \sup_{\delta <|x|<1}  \frac{ 1}{ t^\frac{p}{p-1} |x|^{ (N-1)p-\sigma-2}} \le C R^p +\frac{C}{ t^\frac{p}{p-1} \delta^{(N-1)p-\sigma-2}},\] after noting $ (N-1)p-\sigma-2>0$.  

\item  This estimate comes from applying (\ref{ineq_1}) with $ x= \nabla u_t$ and $ y=\nabla \phi$.    One should note carefully that $ \nabla \phi$ is not small compared to $ \nabla u_t$ (at least away from the origin).    We note generally when applying these fixed point arguments one can take the $ \phi$ term small compared to the $u_t$ term. 
\end{enumerate} 
\end{proof}

\begin{lemma} \label{contraction_lemma} (Contraction) 
\begin{enumerate}
    \item  There is some $C>0$ such that for $ R \in (0,1), t>1, \phi_i \in B_R$ one has 
    \[ \|I \|_Y \le C R^{p-1} \| \phi_2 - \phi_1 \|_X\] where 
    \[ I:= | \nabla u_t+ \nabla \phi_2|^p - p | \nabla u_t|^{p-2} \nabla u_t \cdot \nabla \phi_2- | \nabla u_t+ \nabla \phi_1|^p+ p | \nabla u_t|^{p-2} \nabla u_t \cdot \nabla \phi_1.\] 
    
    \item There is some $C>0$ such that for $ \tau \in (0,1),R>1, \phi_i \in B_\tau$ one has
    \[ \|J \|_Y \le C \left\{ \sup_{|x| \le \delta} |g(x)| + R^{p-1} + \frac{1}{ t \delta^{\alpha-1} }  \right\} \| \phi_2 - \phi_1 \|_X\] where 
    \[ J:= g(x)\left\{ | \nabla u_t + \nabla \phi_2|^p - | \nabla u_t + \nabla \phi_1 |^p \right\}.\] 
\end{enumerate}

\end{lemma}

\begin{proof}  \begin{enumerate} \item    By using (\ref{ineq_2}) with $ x= \nabla u_t, \;  y= \nabla \phi_2, \;  z= \nabla \phi_1$ one can obtain the desired result.  \\

\item  Here we use (\ref{ineq_3}) with $ x= \nabla u_t, \; y= \nabla \phi_2, \; z= \nabla \phi_1$.   Moreover we follow the idea of the proof of Lemma \ref{into_lemma} part 1;  where we consider $ \sup_{|x|<\delta }$ and $ \sup_{\delta <|x|<1}$.
\end{enumerate}

\end{proof}

\noindent
\textbf{Proof of Theorem  \ref{main_non_lin} part 1.}  We now complete the proof of our main theorem.  Recall  we want to find some $\phi$ which satisfies (\ref{eq_phi_intro}) and then $u(x)= u_t(x)+ \phi(x)$ satisfies (\ref{eq_intro}).  We will show that the mapping $T_t$ is a contraction on $B_R$ for suitable $ 0<R<1$ and large $t$.   \\

\noindent
\emph{Into.}  Let $ 0<R<1$, $0<\delta<1$, $ t>1$,  $ \phi \in B_R$ and set $\psi=T_t(\phi)$.   Then by Lemma \ref{into_lemma} there is some $C$ (independent of the parameters) such that 
\[ \| \psi \|_X \le C \left\{ R^p + \sup_{B_\delta} |g| + \frac{1}{ t^\frac{p}{p-1} \delta^{ (N-1)p-\sigma-2}} \right\},\] and hence for $\psi \in B_R$ its sufficient that
\begin{equation} \label{into_cond_first}
C \left\{ R^p + \sup_{B_\delta} |g| + \frac{1}{ t^\frac{p}{p-1} \delta^{ (N-1)p-\sigma-2}} \right\} \le R.
\end{equation} \\

\noindent
\emph{Contraction.}  Let $ 0<R<1$, $0<\delta<1$, $ t>1$,  $ \phi_i \in B_R$ and set $\psi_i=T_t(\phi_i)$.  Then by Lemma \ref{contraction_lemma} we have 
\[ \| \psi_2 - \psi_1 \|_X \le C \left\{ R^{p-1} + \sup_{B_\delta}|g| + \frac{1}{ t \delta^{\alpha-1} } \right\} \|\phi_2 - \phi_1\|_X,\] and hence for $T_t$ to be a contraction its sufficient that 

\begin{equation} \label{contraction_cond_first}
C \left\{ R^{p-1} + \sup_{B_\delta}|g| + \frac{1}{ t \delta^{\alpha-1} }  \right\} \le \frac{1}{2}.
\end{equation}

We now choose the parameters.   Note we see we can satisfy both (\ref{into_cond_first}) and (\ref{contraction_cond_first}) by first taking $R>0$ sufficiently  small,  then taking $ \delta>0$ sufficiently small and then finally taking $ t$ large.   

We now show $u>0$ in $B_1$.    By taking $ R>0$ smaller if necessary we see that $u(x)>0$ for $0<|x|<\E$ for some small $ \E>0$.      We can then apply maximum principle arguments to see that $u>0$ on $\E<|x|<1$.  
\hfill $\Box$

  \section{$ -\Delta u = | \nabla u|^p$ in general bounded  domains; proof of Theorem  \ref{main_non_lin}  part 2} \label{gen_dom_sec}
  
   Without loss of generality we suppose $ 0 \in B_{10 s_0} \subset \Omega \subset \subset B_1$ where $s_0>0$ (for the general case we can perform the needed translation).  Let $ 0 \le \zeta \in C_c^{\infty}(B_{2 s_0})$ with $ \zeta=1$ in $B_{s_0}$,  and   let $ 0 \le \eta \in C_c^\infty(B_{4 s_0})$ with $ \eta =1$ in $B_{2s_0}$ (and both bounded above by $1$). Note $ \zeta \eta = \zeta$.  We look for solutions $u$ of (\ref{eq_intro_2}) of the form $ u(x) = u_t(x) \eta(x)+\phi(x)$   where $ \phi=0$ on $ \pOm$ is the unknown.   Then $u$ is a solution provided $ \phi$ satisfies
      \begin{eqnarray} \label{non_lin_phi_second}
 \left\{ \begin{array}{lcl}
\hfill   L_t(\phi)  &=& u_t \Delta \eta + 2 \nabla \eta \cdot \nabla u_t- \eta | \nabla u_t|^p+| \nabla (u_t \eta)+ \nabla \phi|^p-p| \nabla u_t|^{p-2} \nabla u_t\cdot \nabla\phi  \qquad \mbox{ in }  \Omega \backslash \{0\},  \\
\hfill  \phi &=& 0 \hfill \mbox{ on }  \pOm,
\end{array}\right.
  \end{eqnarray}  
where $L_t$ is as before. \\
We now state our main linear result for $L_t$ on $\Omega$.    Consider the linear problem given by 
  \begin{eqnarray} \label{lin_gen_dom}
 \left\{ \begin{array}{lcl}
\hfill   L_t(\phi)  &=& f \qquad \mbox{ in }  \Omega \backslash \{0\},  \\
\hfill  \phi &=& 0 \hfill \mbox{ on }  \pOm.
\end{array}\right.
  \end{eqnarray}  We define $X$ and $Y$ as the obvious extension of the spaces on the unit ball; 
  \[ \|f\|_Y:= \sup_\Omega |x|^{\sigma+2} |f(x)|, \qquad \| \phi \|_X:=\sup_\Omega \left\{ |x|^\sigma | \phi(x)| + |x|^{\sigma+1} | \nabla \phi(x)| \right\},\] where for the space $X$ we imposed the boundary condition $ \phi=0$ on $ \pOm$. 

\begin{prop}  \label{lin_gen_domain}  There is some $C>0$ and $t_0$ (large) such that for all  $f \in Y$ there is some $ \phi \in X$ which satisfies (\ref{lin_gen_dom}).  Moreover one has the estimate $\| \phi \|_X \le C \|f\|_Y$. 
\end{prop} In the next section we give the proof of this result.   We mention the proof we use utilizes a gluing procedure that  is heavily motivated  by  the approach in \cite {davila}.  \\

\noindent 
\textbf{The fixed point argument.}   We write the first equation in (\ref{non_lin_phi_second}) as 
\[ L_t(\phi)  = \sum_{k=1}^4 I_k \qquad \mbox{ in }  \Omega \backslash \{0\},\] 
where
\[ I_1= u_t \Delta \eta + 2 \nabla \eta \cdot \nabla u_t,\] 
  \[ I_2 = | \nabla (u_t \eta)|^p - \eta | \nabla u_t|^p,\]
  \[I_3 = | \nabla (u_t \eta)+ \nabla \phi|^p- | \nabla (u_t \eta)|^p - p | \nabla (u_t \eta)|^{p-2} \nabla (u_t \eta) \cdot \nabla \phi,\] 
  \[I_4= p\left\{ | \nabla (u_t \eta)|^{p-2} \nabla (u_t \eta) - | \nabla u_t|^{p-2} \nabla u_t \right\} \cdot \nabla \phi.\] 
Now let $t_0>0$ be from Proposition \ref{lin_gen_domain}. For $t>t_0$ define
\[ \E_t:=\sup_{|x|>2 s_0} \left\{ |I_1|+  |I_2|+  \Big|| \nabla (u_t \eta)|^{p-2} \nabla (u_t \eta) - | \nabla u_t|^{p-2} \nabla u_t \Big| \right\}.\]    Using the convergence of $u_t$ and $ \nabla u_t$ to zero away from the origin one sees that $\E_t \rightarrow 0$ as $ t \rightarrow \infty$,  and one can get explicit estimates on $\E_t$, but we won't need them.  \\

\noindent
\emph{Into.}  Let $0<R<1$, $ t>t_0$,  $ \phi \in B_R \subset X$ and set $\psi=T_t(\phi)$.   Then we have 
\[ \| \psi\|_X \le C \sum_{k=1}^4 \|I_k \|_Y \le C_0 \E_t + C_0 \sum_{k=3}^4 \| I_k\|_Y,\] and note one easily sees that 
\[ \|I_4 \|_Y =\sup_{|x|>2s_0} |x|^{\sigma+2} | I_4| \le C_2 \E_t R.\] Using (\ref{ineq_1}) sees that $ \| I_3 \|_Y \le C R^p$.  So we see that for $T_t(B_R) \subset B_R$ its sufficient that 
\begin{equation} \label{into_gen_dom}
C \E_t + C \E_t R + C R^p \le R.
\end{equation} \\

\noindent
\emph{Contraction.}  Let $ 0<R<1$, $ t>t_0$ and $ \phi_i \in B_R$.  Set $ \psi_i= T_T(\phi_i)$ and then note that we have 
\[ \big| L_t(\psi_2 - \psi_1) \big| \le C \left\{ | \nabla \phi_2|^{p-1} + | \nabla \phi_1|^{p-1} \right\} | \nabla \phi_2 - \nabla \phi_1| + \E_t \chi_{ \{|x|>2s_0\}}| \nabla \phi_2 - \nabla \phi_1|,\]  where the first term on right is coming from applying (\ref{ineq_2}) and the second term on the right is coming from the $I_4$ term and $ \chi_A$ is the characteristic function of $A$.    From this we obtain 
\[ \| \psi_2 - \psi_1 \|_X \le \left( C R^p + C \E_t \right) \| \phi_2 - \phi_1\|_X,\] and hence for $T_t$ to be a contraction it is sufficient that 
$R^p + C \E_t \le \frac{1}{2}$.       We now pick the parameters.   By first taking $0<R<1$ sufficiently small and then $t$ large one sees they can easily satisfy the two needed conditions.    

We argue as before to show the solution we get $u$ is indeed singular at the origin and is positive in $ \Omega$.

\subsection{The linear operator $L_t$ on general bounded domains $\Omega$} 

In this section we prove Proposition \ref{lin_gen_domain}. Let $ \zeta,\eta$ denote the cut offs from the previous section.     We look for solutions $ \phi$ or (\ref{lin_gen_dom}) of the form $ \phi(x)= \eta(x) \varphi(x)+\psi(x)$.  Then a computation shows its sufficient that $ \varphi,\psi$ satisfy 
 \begin{eqnarray} \label{lin_gen_varphi}
 \left\{ \begin{array}{lcl}
\hfill   L_t(\varphi)  &=& \zeta f - \frac{ \zeta p x \cdot \nabla \psi}{ \beta |x|^2 + t |x|^{\alpha+1}}\qquad \mbox{ in }  B_1 \backslash \{0\},  \\
\hfill  \varphi &=& 0 \hfill \mbox{ on }  \partial B_1,
\end{array}\right.
  \end{eqnarray} 

\begin{eqnarray} \label{lin_gen_psi}
 \left\{ \begin{array}{lcl}
\hfill -\Delta \psi + \frac{(1-\zeta) p x \cdot \nabla \psi}{ \beta |x|^2+ t |x|^{\alpha+1}}  &=&(1-\zeta)f + \varphi \Delta \eta + 2 \nabla \eta \cdot \nabla \varphi -\frac{p \varphi (x \cdot \nabla \eta )}{ \beta |x|^2+ t |x|^{\alpha+1}}\qquad \mbox{ in }  \Omega \backslash \{0\},  \\
\hfill  \psi &=& 0 \hfill \mbox{ on }  \pOm.
\end{array}\right.
  \end{eqnarray}

As in \cite{davila} we use a fixed point argument to find a solution $(\varphi,\psi)$.   The general procedure is given $ \varphi$ we solve (\ref{lin_gen_psi}) for $ \psi$.    Then we put this $ \psi $ into the right hand side of (\ref{lin_gen_varphi}) and solve for $ \varphi$, which we call $ \hat{\varphi}$.   This defines a nonlinear mapping $ T^t(\varphi) = \hat{\varphi}$ and if we can show this map has a fixed point,  then we have the desired solution (\ref{lin_gen_dom}); of course one still needs the estimate.  \\

\noindent
\textbf{Proof of Proposition \ref{lin_gen_domain}.} Let $t_0$ be from Proposition  \ref{main_linear} and let  $ C_0$ denote the promised constant $C$.   Take $ f \in Y$ with $ \|f\|_Y=1$.     We now will show that $ T^t$ is a contraction mapping on $B_{2C_0} \subset X$ (the closed ball radius $2C_0$ in $X$ centred at the origin).       \\ 

\noindent 
\emph{Into.}  Let $ \varphi \in B_{2C_0}$ and let $ \psi$ satisfy (\ref{lin_gen_psi}).   Note the advection term is zero near the origin and converges uniformly to zero on the $\Omega$.  So by standard elliptic theory there is some $C>0$ such that for all $ t \ge 0$  one has
$\sup_{\Omega} | \nabla \psi| \le C + C C_0$.   Set $ T^t(\varphi)=\hat{\varphi}$.  Then we have 
\[ \| \hat{ \varphi}\|_X \le C_0 \| \zeta f\|_Y + C_0 \big\|\frac{ \zeta p x \cdot \nabla \psi}{ \beta |x|^2 + t |x|^{\alpha+1}} \big\|_Y, \]  and note $ \| \zeta f \|_Y \le 1$ 
 and  the second term is bounded above by 
\[ C \sup_{\Omega} \frac{| \nabla \psi| |x|^{\sigma+1}}{ \beta + t |x|^{\alpha-1}}  \le C \left( C+C C_0 \right) \sup_{\Omega}  \frac{|x|^{\sigma+1}}{ \beta+ t |x|^{\alpha-1}},\] and note \;  $ \delta_t:=\sup_\Omega \frac{|x|^{\sigma+1}}{ \beta+ t |x|^{\alpha-1}} \rightarrow 0$ as $ t \rightarrow \infty$.   So for large enough $t$ we see that $ \hat{\varphi} \in B_{2C_0}$.  \\ 

\noindent 
\emph{Contraction.}  Let $ \varphi_i \in B_{2C_0}$ and we let $ \psi_i$  solve (\ref{lin_gen_psi}) and we set $ \hat{\varphi}_i=T^t( \varphi_i)$.  Using standard estimates and noting the right hand side of (\ref{lin_gen_psi}) is zero near the origin, one sees that 
\[ \sup_\Omega | \nabla \psi_2 - \nabla \psi_1| \le C_1 \| \varphi_2 - \varphi_1 \|_X\] for all $ t \ge 0$.   Then note we have 
\[ L_t( \hat{ \varphi}_2 - \hat{\varphi}_1 ) =  \frac{-p \zeta x \cdot \nabla(\psi_2 - \psi_1)}{ \beta |x|^2 + t |x|^{\alpha+1}} \quad B_1,\] with $\hat{ \varphi}_2 - \hat{\varphi}_1=0$ on $ \partial B_1$.  So as before we get 
\[ \| \hat{ \varphi}_2 - \hat{ \varphi}_1 \|_X \le C \delta_t \sup_{\Omega} | \nabla (\psi_2 - \psi_1)| \le C \delta_t C_1 \| \varphi_2 - \varphi_1\|_X.\] 

So we see for large enough $t$ we can apply Banach's fixed point theorem and obtain a fixed point $ \varphi$, ie. $ T^t(\varphi)= \varphi$.  Moreover note we have $ \| \varphi \|_X \le 2C_0$. Now recall we have $ \phi= \eta \varphi + \psi$.   Using the $X$ bound on $ \varphi$ and the gradient bound on $ \psi$ we see that $ \| \phi \|_X \le C_2$.  This completes the proof of Proposition \ref{lin_gen_domain}. 
\hfill 
$\Box$

\section{Theorem 2; the exterior domain problem} \label{exterior} 
\noindent
\mbox{\textbf{The parameters for the exterior domain problem:}} $N \ge 3$, $p>\frac{N}{N-1}$, $\alpha:=(p-1)(N-1)$, $\beta:= \frac{p-1}{\alpha-1}$,  $\sigma:=N-2+\E,$  where $\E>0$ is small and note $ \alpha>1$. \\

Here we consider 
\begin{eqnarray} \label{eq_intro_ext}
 \left\{ \begin{array}{lcl}
\hfill   -\Delta u    &=& | \nabla u|^p  \qquad \mbox{ in }  \Omega,  \\
\hfill u &=& 0 \hfill \mbox{ on }  \pOm,
\end{array}\right.
  \end{eqnarray} in the case 
where  $\Omega$ is an exterior domain (with smooth boundary) in $ \IR^N$ with $N \ge 3$ and $ \frac{N}{N-1}<p<2$.  We show there is a positive classical solution of (\ref{eq_intro}).    For simplicity we assume that $B_2^c \subset \subset \Omega \subset \subset B_1^c$ where the $c$ denotes compliment.   We begin by looking at an explicit example the the compliment of the unit ball.

\begin{exam} \label{exam_ext} Let the parameters be as above and set
\[ u_t(r) = \int_1^r \frac{1}{ \left( ty^\alpha - \beta y\right)^\frac{1}{p-1}} dy.\] Then for all $ t>\beta$, $u_t$ is a classical positive solution of (\ref{eq_intro}) in the case of $ \Omega= \overline{B_1}^c$.     Also note that $u_t$ is increasing in $r$ and is bounded.   Also we see that $u_t, \nabla u_t$ converge uniformly to zero on $\overline{B_1}^c$ as $ t \rightarrow \infty$. 
\end{exam}   For notational convenience now, when solving a pde on a ball or an exterior of a ball we will write $ B_r$ or $B_r^c$; of course its understood the domain is always open. 
As in the case of bounded domain $ \Omega$ we will look for a solution of the form $ u(x)= \eta(x) u_t(x) + \phi(x)$ where $\eta$ is a suitable cut off to make $u=0$ on $ \pOm$;   take  $0 \le \eta \le 1$ to be smooth with $ \eta=0$ in $B_2$ and $ \eta=1$ for $ B_3^c$.    As before the linearized operator will be of crucial importance.  We set   
\[ L^t(\phi):=\Delta \phi + p | \nabla u_t|^{p-2} \nabla u_t \cdot \nabla \phi,\]  and an explicit computation shows 
\[ L^t(\phi)= \Delta \phi + \frac{p x \cdot \nabla \phi(x)}{|x| \left( t |x|^\alpha - \beta |x| \right)}.\]      We now choose our function spaces.  As before we define 
\[ \| \phi\|_X:=\sup_\Omega |x|^{\sigma+1} | \nabla \phi(x)|, \qquad \|f\|_Y:=\sup_\Omega |x|^{\sigma+2} |f(x)|,\] where $\sigma$ is to be determined and where the spaces $X$ and $Y$ are defined using the above norms;  the space $X$ we impose $ \phi=0$ on $ \pOm$.  \\

\noindent 
\textbf{The parameter $\sigma$.}  As before we will employ a fixed point argument to obtain $ \phi \in B_R:=\{ \phi \in X: \| \| \phi \|_X \le R\}$ where $R>0$ is small, and where $ u(x)=\eta(x) u_t(x) + \phi(x)$ is a solution.   The order in choosing the parameters will be the same as before;  we will pick $R>0$ small and then take $t$ large.  Recalling that $u_t$ (and its derivatives in $ x$) converge to zero when $ t \rightarrow \infty$ we see there will be a natural hurdle of showing $ u  \neq 0$;  this was not an issue in the previous results since no matter how large $t$ was chosen we had uniform blow up near the origin.  
So returning to the form of our solution we see that  if $ \phi \in B_R$   and $|x|$ large we have 
\[ | \nabla u(x)| \ge \frac{1}{\left( t r^\alpha- \beta r \right)^\frac{1}{p-1}} - \frac{R}{|x|^{\sigma+1}},\] where $ r=|x|$.    From this we see no matter how large $t$ is chosen (or the value of $R$) that if $ \sigma+1>\frac{\alpha}{p-1}$  then for large enough $|x|$ we have $ \nabla u(x) \neq 0$.  With this in mind we choose $ \sigma:=N-2+\E$ where $\E>0$ is small.   One should note that this value of $\sigma$ is  somewhat nonstandard.  A lot of linear theory has been done where $ \sigma \in (0,N-2)$.     Typically the $X$ norm would also have a zero order term given by $ |x|^{\sigma} | \phi(x)|$;  for our value of $\sigma$ we cannot include this term  but this doesn't affect us since we really only need decay of the gradients.     In the next section we will show the desired linear theory that there is some $C>0$ and large $t_0$ such that for all $t>t_0$, for all $ f \in Y$ there is some $ \phi \in X$ which satisfies $ L^t(\phi)=f$ in $\Omega$ with $ \phi=0$ on $ \pOm$.   Moreover one has $ \| \phi \|_X \le C \|f\|_Y$.

\begin{remark} We remark that in our first attempt at proving the needed linear theory for $L^t$ we used a proof similar to the previous sections.   We first considered the result on $B_1^c$ using spherical harmonics and a blow up argument.  We then used the gluing procedure from the previous section to extend this to a general exterior domain.    The result held for all $t$ in the allowed range (except $t$ had to be bounded away from $\beta$).  Later we realized that for large $t$ (and we really only need the result for large $t$) that $L^t$ is really just a perturbation of the Laplacian and hence we can prove the needed result via a more abstract approach.    It is still useful to consider the spherical harmonic approach on $B_1^c$ to see exactly how the zero order estimate fails.   
\end{remark}

\noindent 
\textbf{The nonlinear set up and the fixed point argument.}  Here we follow the general procedure as in the case of a general bounded domain $ \Omega$.  A computation shows that $u$ (as described above) is a solution of (\ref{eq_intro_ext}) if $ \phi$ satisfies 

      \begin{eqnarray}  \label{non_lin_phi_ext}
 \left\{ \begin{array}{lcl}
\hfill   -L^t(\phi)  &=& u_t \Delta \eta + 2 \nabla \eta \cdot \nabla u_t- \eta | \nabla u_t|^p+| \nabla (u_t \eta)+ \nabla \phi|^p-p| \nabla u_t|^{p-2} \nabla u_t\cdot \nabla\phi  \qquad \mbox{ in }  \Omega \backslash \{0\},  \\
\hfill  \phi &=& 0 \hfill \mbox{ on }  \pOm,
\end{array}\right.
  \end{eqnarray} and as before we rewrite this as  
\[ -L^t(\phi)  = \sum_{k=1}^4 I_k \qquad \mbox{ in }  \Omega \backslash \{0\},\] 
where
\[ I_1= u_t \Delta \eta + 2 \nabla \eta \cdot \nabla u_t, \quad I_2 = | \nabla (u_t \eta)|^p - \eta | \nabla u_t|^p,\]
  \[I_3 = | \nabla (u_t \eta)+ \nabla \phi|^p- | \nabla (u_t \eta)|^p - p | \nabla (u_t \eta)|^{p-2} \nabla (u_t \eta) \cdot \nabla \phi,\] 
  \[I_4= p\left\{ | \nabla (u_t \eta)|^{p-2} \nabla (u_t \eta) - | \nabla u_t|^{p-2} \nabla u_t \right\} \cdot \nabla \phi.\]   Note that $I_1=I_2=0$ in $B_2 \cup B_3^c$.  Also we have $I_4=0$ in $B_3^c$.   Similiar to before we set 
  \[\E_t:=\sup_{\Omega} \left\{ |I_1|+  |I_2|+  \Big|| \nabla (u_t \eta)|^{p-2} \nabla (u_t \eta) - | \nabla u_t|^{p-2} \nabla u_t \Big| \right\},\]  and note this is really a sup of $\Omega \cap B_3$ and hence its trivial to see $ \E_t \rightarrow 0$ as $ t \rightarrow \infty$ after taking into account the behaviour of $u_t$ for large $t$.      We now consider the fixed point argument.   Consider $T^t(\phi)=\psi$ where  
  \[ -L^t(\psi)= \sum_{k=1}^4 I_k(\phi) \;  \mbox{ in }  \Omega, \quad  \psi =0 \; \mbox{ on } \pOm,\]  where we are writing $I_k(\phi)$ to indicate the $ \phi$ dependence. \\ 
  
  \noindent
  \emph{Into.} Let $ 0<R<1$, $ t>t_0$, $ \phi \in B_R \subset X$ and $ \psi=T^t(\phi)$.  Then we have 
  \[ \| \psi\|_X \le C\E_t \left(1+R \right) + C \|I_3 \|_Y,\] where this last term will depend on whether $ p \le 2$ or $ p >2$.   We first consider the case of $ p \le 2$; and in this case we use (\ref{ineq_1}) with $ x= \nabla( u_t \eta)$, $y=\nabla \phi$ to arrive at 
  \[ \|I_3 \|_Y \le C \sup_\Omega |x|^{\sigma+2} | \nabla \phi|^p \le C R^p \sup_{\Omega} |x|^{\sigma+2 -p(\sigma+1)}, \qquad \mbox{(case $p \le 2$)},\] which is bounded by $CR^p$ provided $\sigma+2-p(\sigma+1) \le 0$ which is in fact the case after recalling the value of $\sigma=N-2+\E$.  For $p >2$ we will use the following inequality
  \[ \Big| |x+y|^p - |x|^p - p |x|^{p-2} x \cdot y \Big| \le C |y|^p + C |x|^{p-2} |y|^2, \qquad x,y \in \IR^N,\] and after taking $x$ and $y$ as above gives 
  \[ \|I_3 \|_Y \le C R^p + CR^2\sup_\Omega |x|^{\sigma+2} | \nabla(u_t \eta)|^{p-2} |x|^{-2\sigma-4}.\] Considering the convergence to zero of $\nabla u_t$ and $u_t$ we see the only possible issue of the second term is for large $x$. For large $x$ note that 
  \[ | \nabla (u_t \eta)|^{p-2} \le  \frac{C}{ t^\frac{p-2}{p-1} |x|^{(N-1)(p-2)}}.\] Using this we can substiture into the above (after taking into account the value of $\sigma$) to arrive at:   there is some $\hat{\E}_t \rightarrow 0$ as $ t \rightarrow \infty$ such that 
  \[ \|I_3 \|_Y \le C R^p + C \hat{\E}_t R^2, \qquad \mbox{(case $p>2$).}\]  So for $T^t(B_R) \subset B_R$ (in either case) its sufficient that 
  \begin{equation} \label{into_ext_form}
  C \left( \E_t(1+R) + R^p +\hat{\E}_t R^2 \right) \le R.
  \end{equation}  \\

    \noindent
  \emph{Contraction.} Let $ 0<R<1$, $ t>t_0$, $ \phi_i \in B_R \subset X$ and $ \psi_i=T^t(\phi_i)$.  Then note we have 
  \[ |I_4(\phi_2)-I_4(\phi_1)| \le \E_t \chi_{B_3}(x) | \nabla \phi_2 - \nabla \phi_1|,\] and hence $ \| I_4(\phi_2)-I_4(\phi_1)\|_Y \le C \E_t \| \phi_2 - \phi_1 \|_X.$  To examine the $I_3$ term we use (\ref{ineq_3}) with $ x= \nabla (u_t \eta)$ and $y=\nabla \phi_2, z=\nabla \phi_1$ to arrive at 
  \[ |I_3(\phi_2)-I_3(\phi_1)|\le C \left\{ | \nabla(u_t \eta)|^{p-1} + | \nabla \phi_2|^{p-1} + | \nabla \phi_1|^{p-1} \right\} | \nabla \phi_2 - \nabla \phi_1|.\]  Using these estimates we see 
  \[ \|I_3(\phi_2)-I_3(\phi_1)\|_Y \le C \left( R^{p-1}+ \sup_\Omega |x|| \nabla (u_t \eta)|^{p-1} \right) \| \phi_2 - \phi_1 \|_X.\] A computation shows that 
  \[ | \nabla (u_t \eta)|^{p-1} \le \frac{C}{t |x|^\alpha} \quad \mbox{ in } \Omega,\] for large $t$. Using this and the fact that $ \alpha>1$ we see that 
  \[  \|I_3(\phi_2)-I_3(\phi_1)\|_Y \le C \left( R^{p-1}+ \frac{1}{t} \right) \| \phi_2 - \phi_1 \|_X,\]  and hence for $T^t$ to be a contraction on $B_R$ its sufficient that 
  \begin{equation} \label{contr_exter}
   C \left( R^{p-1}+ \frac{1}{t} \right) \le \frac{1}{2}.
   \end{equation}   We now choose the parameters $ R$ and $t$.  By taking $ R>0$ sufficiently small and fixing and then taking $ t$ large we see that we can satisfy (\ref{into_ext_form}) and (\ref{contr_exter}).  We can now apply the Banach's fixed point theorem to see there is some $ \phi \in B_R$ such that $T^t(\phi)=\phi$.   As noted earlier for large $x$ we have $ \nabla u(x) \neq 0$ and hence we know $u$ is not identically zero.  Also note that a computation shows 
   \[ \nabla u(x) \cdot x \ge \frac{1}{r^{N-2}} \left( \frac{1}{ ( t -\beta r^{1-\alpha})^\frac{1}{p-1}} - \frac{R}{r^\E} \right),\] where $ r=|x|$.  So we see for large $|x|$ that $u$ is increasing in the radial direction.   Now we show $u$ is positive.  Suppose not,  then using the monotonicity in the radial direction we see there is some $x_0 \in \Omega$ such that $ \min_\Omega u=u(x_0) \le 0$.   Then we can use the strong maximum principle to see that $ u=u(x_0)$ in $ \Omega$; a contradiction.

\subsection{The linear theory}

We begin with a theorem regarding the mapping properties of the Laplacian and for this we need to define a new norm.  Consider 
\[ \| \phi \|_{\widehat{X}}:=\sup_{\Omega} \left\{ |x|^{\sigma+1} | \nabla \phi(x)| + |x|^{\sigma+2} | \Delta \phi(x)| \right\},\] and we set $\widehat{X}:=\{ \phi: \| \phi \|_{\widehat{X}}<\infty,  \mbox{ and } \phi=0  \mbox{  on } \partial \Omega \}$. 

\begin{thm} \label{laplace} The mapping $ \Delta: \widehat{X} \rightarrow Y$ is continuous, linear, one to one and onto with continuous inverse. 
\end{thm}
As a corollary of this will obtain results regarding the 
 solvability of 
\begin{eqnarray} \label{Lin_eq_gen}
 \left\{ \begin{array}{lcl}
\hfill   L^t (\phi)(x)    &=& f(x) \qquad \mbox{ in }  \Omega,  \\
\hfill \phi &=& 0 \hfill \mbox{ on }  \pOm.
\end{array}\right.
  \end{eqnarray}

\begin{cor} \label{general_domain_prop} There is some $t_0$ large and $C$ such that for all $ t >t_0$ and for all $f \in Y$ there is some $\phi \in X$ that satisfies (\ref{Lin_eq_gen}).  Moreover $\| \phi \|_X \le C \|f\|_Y$.  
\end{cor}

\begin{lemma} \label{kernel_lap} (Kernel of $\Delta$) Suppose $ \Delta \phi=0$ in $ \IR^N \backslash \{0\}$ with $\sup_{0<|x|} |x|^{\sigma+1} | \nabla \phi(x)|<\infty$  or $ \Delta \phi=0$ in $ B_1 \backslash \{0\}$ with $ \partial_\nu \phi=0$ on $ \partial B_1$ and $ \sup_{B_1} |x|^{\sigma+1} |\nabla \phi(x)|<\infty$.   Then $ \phi$ is a constant. 
\end{lemma}

\begin{proof}     Now suppose $ \phi$ as in the hypothesis and we write as $ \phi(x)=\sum_{k=0}^\infty a_k(r) \psi_k(\theta)$.    Then for all $ k \ge 0$ we have 
   \begin{equation} \label{ode_lap}
   a_k''(r)+\frac{N-1}{r} a_k'(r)- \frac{\lambda_k a_k(r)}{r^2} = 0, \quad \mbox{ in } r \in (0,R),
   \end{equation}  where $R=\infty$ in the first case and in the second case $R=1$ and one has the boundary condition $a_k'(1)=0$.  Also note there is some $C_k>0$ such that we have $ \sup_{0<r<R} r^{\sigma+1} |a_k'(r)| <\infty$.  We now consider the various modes. 
   
   \begin{itemize}
       \item ($k=0$). The general solution in this case is $ a_0(r)= C_0 + \frac{D_0}{r^{N-2}}$. We first consider  the case of $R=\infty$.  In this case we see to satisfy the gradient bound we must have $D_0=0$ and hence $a_0$ is constant.     When $R=1$ we also see $D_0=0$ since $a'_0(1)=0$.   
       
       \item ($k \ge 1$). Note $a_k$ satisfies an ode of Euler type and hence the roots of $ \gamma^2 +(N-2)\gamma-\lambda_k=0$ are relevant.  In this case the general solution is given by $ a_k(r)=C_k r^{\gamma_k^+} + D_k r^{\gamma_k^-}$ where 
       \[ \gamma_k^\pm = \frac{-(N-2)}{2} \pm  \frac{\sqrt{ (N-2)^2+4 \lambda_k} }{2}.\]  Note that 
       \[ r^{\sigma+1} a_k'(r)= C_k \gamma_k^+ r^{\gamma_k+\sigma}+ D_k \gamma_k^- r^{\gamma_k^-+\sigma}.\] In the case of $ R=\infty$ we see that if $ \gamma_k^+,\gamma_k^-, \gamma_k^++\sigma, \gamma_k^-+\sigma$ are all nonzero then we can show the quantity on the left is unbounded in $r$ unless $ C_k=D_k=0$.   We come back to these verifying these quantities are nonzero shortly.   Now consider the case of $R=1$.  Then to satisfy $a_k'(1)=0$ imposes the condition $ C_k \gamma_k^+ + D_k \gamma_k^-=0$ and hence 
       \[ r^{\sigma+1} a_k'(r)= C_k \gamma_k^+ \left( r^{\gamma_k^++\sigma}-r^{\gamma_k^-+\sigma} \right).\]   In this case note if $ \gamma_k^-+\sigma<0$  then we must have $C_k=0$ to satisfy the desired estimate.  \\
       We now consider the various parameters in question.  Note that $\gamma_1^+=1$ and $\gamma_1^-=-N+1$ and hence by monotonicity of $ \gamma_k^\pm$ we see $ \gamma_k^\pm \neq 0$ for $k \ge 1$. Also note by montonicity we have 
       \[ \gamma_k^+ +\sigma \ge \gamma_1^++\sigma = \sigma+1>0, \qquad \gamma_k^-+\sigma \le \gamma_1^-+\sigma =-1+\E<0.\] 
       \end{itemize}
\end{proof}

\noindent 
\textbf{Proof of Theorem \ref{laplace}.}  Its clear $ \Delta $ is linear and continuous (to see its continuous note the $\widehat{X}$ norm includes the graph norm for $\Delta$).     \\ 

\noindent \emph{One to one.}  Let $ \phi \in \widehat{X}$ with $ \Delta \phi=0$ in $\Omega$ and $ \phi=0$ on $ \pOm$.  By integrating the first order portion of the $\widehat{X}$ norm along a ray one sees that $ \phi$ is bounded.    Let $R$ be big and multiply the equation by $ \phi$ and integrate over $\Omega \cap B_R$ (the open ball centred at the origin in $\IR^N$) to see 
\[ \int_{\Omega \cap B_R} | \nabla \phi|^2 dx \le \sup_{\Omega} | \phi| \int_{\partial B_R} | \nabla \phi| \le  \sup_\Omega | \phi|  C_N \| \phi \|_{\widehat{X}} R^{-\E},\] after recalling the value of $\sigma$.  Sending $ R \rightarrow \infty$ we see that $ \phi=0$ after taking into account the boundary condition of $ \phi$.   \\

\noindent
\emph{Onto.}  Let $R_m \rightarrow \infty$ and consider the problem

\begin{eqnarray} \label{Lin_eq_approx}
 \left\{ \begin{array}{lcl}
\hfill  \Delta \phi(x)    &=& f(x) \qquad \mbox{ in }  \Omega_m,  \\
\hfill \phi &=& 0 \hfill \mbox{ on }  \pOm,  \quad \mbox{ \footnotesize(the inner boundary)}\\
\hfill \partial_\nu \phi &=& 0 \quad \hfill \mbox{ on }  \partial B_{R_m},  \quad \mbox{ \footnotesize(the outer boundary)}\\
\end{array}\right.
  \end{eqnarray} where $\Omega_m:= \Omega \cap B_{R_m}$. We claim there is some $C>0$  such that for all $m$ large and $ f_m \in Y$ there is some $ \phi_m$ which satisfies (\ref{Lin_eq_approx}) and moreover one has the estimate $ \| \phi_m \|_{\widehat{X}} \le C \|f_m\|_Y$.  We accept the validity of the claim for now.    Then given $ f \in Y$ (on $\Omega$) we let $ \phi_m$ satisfy (\ref{Lin_eq_approx}).   We can then use a diagonal argument and compactness to pass to the limit (after passing to a suitable subsequence) to find some $\phi \in \widehat{X}$ (on $\Omega$) which satisfies $ \Delta \phi=f$ in $\Omega$ with $ \phi=0$ on $ \pOm$.  Moreover one has the estimate $ \| \phi \|_{\widehat{X}} \le C \|f\|_Y$.    \\
  
  \noindent
  \emph{Proof of claim.} There is no issue with the existence of a solution of (\ref{Lin_eq_approx}),  the only possible problem is the estimate fails.  We first prove the estimate if we replace the $\widehat{X}$ norm with the $X$ norm. Towards a contradiction we can assume for large enough $m$ the estimate fails.  Then after normalizing there is $ \phi_m \in {X}$ (in $\Omega_m$) and $ f_m \in Y$ (in $\Omega_m$) which satisfies (\ref{Lin_eq_approx}) and $ \| \phi_m \|_{X}=1$ and $ \|f_m\|_Y \rightarrow 0$. Then there is some $ x_m \in \Omega_m$ such that $ |x_m|^{\sigma+1} | \nabla \phi_m(x_m)| \ge \frac{1}{2}$ and we set $ s_m=|x_m|$.   We consider three cases: 
   (i) $ s_m $ bounded; \quad (ii) $ \frac{s_m}{R_m} \rightarrow 0$; \quad   (iii) $ \frac{s_m}{R_m}$  bounded away from zero.  \\
   
   \noindent 
   \emph{Case (i).}  Using compactness and a diagonal argument we see that there is some $ \phi $ such that $\phi_m \rightarrow \phi $ in $ C^{1,\delta}_{loc}(\overline{\Omega} \cap B_R)$ for all $R$ large.  Using the convergence we can pass to the limit in the equation and hence $ \Delta \phi=0$ in $ \Omega$ with $ \phi=0$ on $ \pOm$. Also note that since $ s_m$ is bounded there is some $x_0 \in \overline{\Omega}$ such that $ | \nabla \phi(x_0)| \ge \frac{1}{2}$.  Additionally we have $ | \nabla \phi(x)| \le 1$ in $\Omega$ and hence we can apply our result regarding the kernel to obtain a contradiction.  \\

   \noindent 
   \emph{Case (ii).}  In this case we consider $ \zeta_m(z):= s_m^\sigma \phi_m(s_m z)$ for $z \in \Omega^m:=\{z:  s_m z \in \Omega_m \}$ and note $\Omega_m \rightarrow \IR^N \backslash \{0\}$.  We define $z_m$ by $s_m z_m=x_m$ and hence $ | z_m|=1$ satisfies $ | \nabla \zeta_m(z_m)| \ge \frac{1}{2}$.  Additionally note that $|z|^{\sigma+1} | \nabla \zeta_m(z)| \le 1$ in $ \Omega^m$.   Set $\zeta^m(z):= \zeta_m(z)-\zeta_m(z_m)$ and hence $ \zeta^m$ satisfies the same estimates as $\zeta_m$ and $ \zeta^m(z_m)=0$.   Also we note that $\Delta \zeta^m(z)=g_m(z):=s_m^{\sigma+2} f_m(s_m z)$ in $ \Omega^m$.   Using a diagonal and compactness argument there is some $ \zeta$ such that $ \Delta \zeta=0$ in $ \IR^N \backslash \{0\}$;  $ \zeta^m \rightarrow \zeta$ in $C^{1,\delta}_{loc}( \IR^N \backslash \{0\})$ and if $z_m \rightarrow z_0$  then we have $|\nabla \zeta(z_0)| \ge \frac{1}{2}$.  But this contradicts the results from Lemma \ref{kernel_lap}. \\

   \noindent 
   \emph{Case (iii).} We now assume $\frac{s_m}{R_m}$ is bounded away from zero.  Here we consider $ \zeta_m(z):=R_m^\sigma \phi_m(R_m z)$ for $ z \in \Omega^m:=\{ z: R_m z \in \Omega_m \}$ and note the outer portion of the boundary of $\Omega^m$ is just $ \partial B_1$. Also note $\Omega_m$ is roughly an annulus with a shrinking hole at the origin. We define $z_m$ by $ R_m z_m=x_m$ and so $ |z_m| \le 1$ and is bounded away from zero.    Also note we have $ | \nabla \zeta_m(z)| \le 1$ in $ \Omega^m$ and $ | \nabla \zeta_m(z_m)| \ge \E_0$ for some $\E_0>0$.   We now set $ \zeta^m(z)=\zeta_m(z)-\zeta_m(z_m)$ and note $\zeta^m$ satisfies the same estimates and $\zeta^m(z_m)=0$.   Also note that $ \zeta^m$ satisfies $\Delta \zeta^m(z)= g_m(z):=R_m^{\sigma+2} f_m(R_m z)$ in $\Omega^m$  with $ \partial_\nu \zeta^m =0$ on $\partial B_1$ (the outer portion of $\partial \Omega^m$) and we omit the boundary condition on the inner boundary.  By a compactness argument and a diagonal argument there is some $\zeta$ such that $\zeta^m \rightarrow \zeta$ in $C^{1,\delta}_{loc}( \overline{B_1} \backslash \{0\})$.   Moreover we have $| \nabla \zeta| \neq 0$ and $ |z|^{\sigma+1} | \nabla \zeta(z)| \le 1$ in $B_1 \backslash \{0\}$ and $ \Delta \zeta(z)=0$ in $B_1 \backslash \{0\}$ with $ \partial_\nu \zeta=0$ on $ \partial B_1$.  But this contradicts the results from Lemma \ref{kernel_lap}.   \\
   
   So we have shown that we have the desired gradient estimate on $\phi$.  The second order estimate on $\phi$ comes directly off the equation. 
   
   \hfill $\Box$


\noindent
\textbf{Proof of Corollary \ref{general_domain_prop}.} Recall we have 
\[ L^t(\phi)= \Delta \phi + \frac{p x \cdot \nabla \phi(x)}{|x| \left( t |x|^\alpha - \beta |x| \right)}.\]  The claim is that for large $t$ we can see $L^t$ as a perturbation of $\Delta$.   To see this we write $ \delta=\frac{1}{t}$ and then we can write 
\[ L^t(\phi)= \tilde{L}^\delta:=\Delta \phi + T^\delta(\phi),\] where 
\[ T^\delta(\phi)(x):= \frac{\delta p x \cdot \nabla \phi(x)}{|x| \left( |x|^\alpha -\delta \beta |x| \right)}.\]   Let $ \phi \in X$ with $ \| \phi \|_X \le 1$ and then note we have 
\[ \|T^\delta(\phi) \|_Y \le \delta p \sup_\Omega \frac{1}{|x|^{\alpha-1} -\delta \beta},\] for small enough $ \delta$ and hence the operator norm $ \|T^\delta\|_{\mathcal{L}(X,Y)} \le C \delta$ for small enough $ \delta$.  Using this and Theorem \ref{laplace} one can apply some standard functional analysis to complete the proof.     Note if one tries a similar argument on the linear operators from the previous sections they will see it fails. 
\hfill $\Box$.

\end{document}